\documentclass{amsart}
\usepackage[foot]{amsaddr}
\makeatletter
\renewcommand{\email}[2][]{%
  \ifx\emails\@empty\relax\else{\g@addto@macro\emails{,\space}}\fi%
  \@ifnotempty{#1}{\g@addto@macro\emails{\textrm{(#1)}\space}}%
  \g@addto@macro\emails{#2}%
}
\makeatother

\makeatletter
\@namedef{subjclassname@2020}{\textup{2020} Mathematics Subject Classification}
\makeatother

\usepackage{amsthm,amsmath,amssymb,enumerate,hyperref,fontenc,graphicx}
\usepackage{algorithm,algpseudocode}

\usepackage{bbm}
\usepackage{multirow}
\AtBeginDocument{
  \renewcommand{\setminus}{\mathbin{\backslash}}%
}
\usepackage{stmaryrd}
\usepackage{tikz,tikz-cd}
\usetikzlibrary{matrix,arrows,decorations.markings,calc}

\usepackage{pgfplots}
\pgfplotsset{
        compat=1.16,
    }
    
\theoremstyle{plain}
\newtheorem{theorem}{Theorem}[section]
\newtheorem{lemma}[theorem]{Lemma}

\theoremstyle{definition}
\newtheorem{definition}[theorem]{Definition}
\newtheorem{example}[theorem]{Example}
\newtheorem{remark}[theorem]{Remark}

\newtheorem{problem}[theorem]{Problem}

\newcommand{\mbbR}{{\mathbb R}}
\newcommand{\mbbZ}{{\mathbb Z}}

\newcommand{\mcalC}{{\mathcal C}}

\newcommand{\mfrako}{{\mathfrak o}}

\newcommand{\va}{{\mathbf a}}
\newcommand{\vb}{{\mathbf b}}
\newcommand{\ve}{{\mathbf e}}

\newcommand{\vE}{{\mathbf E}}

\DeclareMathOperator{\ab}{\operatorname{ab}}

\DeclareMathOperator{\ec}{\operatorname{ec}}

\DeclareMathOperator{\supp}{\operatorname{supp}}

\DeclareMathOperator{\concat}{\operatorname{concat}}
\DeclareMathOperator{\Out}{\operatorname{Out}}
\DeclareMathOperator{\In}{\operatorname{In}}
\DeclareMathOperator{\red}{\operatorname{red}}

\begin{document}

\title{Shortest Circuits in Homology Classes of Graphs}
%
%
%

\author{Ye Luo} 
\address[A1]{School of Informatics, Xiamen University, Xiamen, Fujian 361005, China}
\email[A1]{luoye@xmu.edu.cn}


\date{}
\begin{abstract}
Recently, the study of circuits and cycles within the homology classes of graphs has attracted considerable research interest. However, the detection and counting of shorter circuits in homology classes, especially the shortest ones, remain underexplored. This paper aims to fill this gap by solving the problem of detecting and counting the shortest cycles in homology classes, leveraging the concept of direction-consistent circuits and extending classical results on Eulerian circuits such as Hierholzer's algorithm and the BEST theorem. As an application, we propose the one-carrier transportation routing problem and relate it to a circuit detection problem in graph homology. 
\end{abstract}
\maketitle
\section{Introduction}
In graph theory, the homology groups of graphs~\cite{S2013homology}  are used to quantify and classify topological features within graph. In  a recent work by \cite{LR2024spectral}, the authors explored properties of circuits and cycles in prescribed homology classes of a graph, presenting some trace formulas for counting circuits and cycles in these homology classes and some asymptotic formulas for this counting problem as the circuit length approaches infinity. In this paper, we will take a different approach by investigating shorter circuits, specifically focusing on the detecting and counting problems of the shortest circuits within the homology classes of graphs. 

Let us introduce some preliminary terminologies and notations. Throughout this paper, a \emph{graph} refers to a (not necessarily connected) finite graph that may contain loops and multiple edges but has no isolated vertices, and we say a graph is a \emph{simple graph} if it does not contain any loops or multiple edges.  Let $G$ be a graph with vertex set $V(G)=\{v_1,\cdots,v_n\}$ and the (unoriented) edge set of $G$ be $E(G)=\{e_1,\cdots,e_m\}$. 
Consider an  \emph{orientation} $\mathfrak{o}$  of $G$. Let $\ve_{\mathfrak{o},i}$ and $\ve_{\mathfrak{o},m+i}=\ve_{\mathfrak{o},i}^{-1}$ be respectively the positively and negatively oriented edges corresponding to $e_i$ with respect to $\mathfrak{o}$. Let $\vE_\mathfrak{o}(G)=\{\ve_{\mathfrak{o},1},\cdots,\ve_{\mathfrak{o},m}\}$ and  $\vE(G)=\{\ve_{\mathfrak{o},1},\cdots,\ve_{2m}\}$. Then $G_\mathfrak{o}:=(V(G),\vE_\mathfrak{o}(G))$ is a digraph whose underlying graph is $G$.  For each oriented edge $\ve\in\vE(G)$, denote by $\ve(0)$ and $\ve(1)$ the \emph{initial} and \emph{terminal} vertices of $\ve$ respectively. Then a \emph{walk} $P$ of length $l$ on $G$ is a sequence of oriented edges $\va_1\cdots\va_L$ such that $\va_{i+1}(0)=\va_{i}(1)$ for $i=1,\cdots,l-1$, and we call $P(0):=\va_1(0)$ and $P(1):=\va_l(1)$ the \emph{initial} and \emph{terminal} vertices of $P$ respectively. If the initial and terminal vertices of $P$ are the same vertex $v$, then we say $P$ is a \emph{closed walk} at the base vertex $v$. In the following, we will mainly consider closed walks with no backtracks and tail (meaning that $P$ can not get through $\ve$ and $\ve^{-1}$ consecutively), which are called \emph{circuits} on $G$.   Denote by $\mcalC_G$ be the set of all circuits on $G$. Then for each $k\in\mbbZ$, there is a translation map $\mathfrak{t}_k$ acting on $\mcalC_G$ by sending a circuit $C=\va_0\cdots\va_{l-1}$ to $\mathfrak{t}_k(C) = \va_{s_0}\cdots \va_{s_{l-1}}$ where $s_i  =( i+k) \mod L$. Clearly the action of all translations induces an equivalence relation $\sim$ on $\mcalC_G$. 
Denote by $[C]=(\va_0\cdots\va_{L-1})$ the equivalence class of circuits equivalent to $C$, which is also called the \emph{cycle} corresponding to the circuit $C$.

Let $C_1(G,\mbbZ)$ be the free abelian group on $\vE_\mfrako$, which can be equivalently considered as a $\mbbZ$-module on $\vE$ with $1\cdot \ve^{-1} = (-1)\cdot \ve$ for all $\ve\in \vE$. For $\alpha = \sum_{\ve\in \vE_\mfrako(G)}c_\ve\cdot \ve\in C_1(G,\mbbZ)$ with respect to any orientation $\mfrako$, the degree of $\alpha$ at $\ve\in\vE(G)$, denoted by $\deg_\alpha(\ve)$, is defined by $\deg_\alpha(\ve)=c_\ve$ if $\ve\in \vE_\mfrako(G)$ and $\deg_\alpha(\ve)=-c_\ve$ if $\ve\in \vE(G)\setminus\vE_\mfrako(G)$. For each $v\in V(G)$, let $\deg^+_{\mfrako,\alpha}(v) = \sum_{\ve\in \vE_\mfrako(G), \ve(0)=v}\deg_\alpha(\ve)$, $\deg^-_{\mfrako,\alpha}(v) = \sum_{\ve\in \vE_\mfrako(G), \ve(1)=v}\deg_\alpha(\ve)$ and $\deg_\alpha(v) = \sum_{\ve\in \vE(G),\ve(0)=v}\deg_\alpha(\ve)$. Note that  $\deg_\alpha(v)=\deg^+_{\mfrako,\alpha}(v) - \deg^-_{\mfrako,\alpha}(v)$ for any orientation $\mfrako$ on $G$.

\begin{definition} \label{D:homology}
We say  $\alpha \in C_1(G,\mbbZ)$ is a \emph{circulation} or a \emph{homology class} if  for each vertex $v\in V$, the \emph{balancing condition}  $\deg_\alpha(v)=0$  is satisfied. The set of all circulations form a group, called  the \emph{first homology group} of $G$ and denoted by $H_1(G,\mbbZ)$.
The  \emph{abelianization} of a path $P=\va_1\cdots \va_l$ is defined as   $P^{\ab}:=\va_1+\cdots+\va_l \in C_1(G,\mbbZ)$. 
\end{definition}

A \emph{metric} on $G$ is a function $\mu: E(G)\to \mbbR_{>0}$. We may also treat $\mu$ as a function on  $\vE(G)$ with $\mu(\ve)=\mu(\ve^{-1}):=\mu(e)$ for all $\ve\in\vE(G)$ and its underlying edge $e\in E(G)$. 
The \emph{$\mu$-length} of a walk $P=\va_1\cdots\va_l$ is straightforwardly defined as $\tau_\mu(P):=\sum_{i=1}^l \mu(\va_i)$. If $\mu$ is trivial (assigning $1$ to each edge), then we simply call the $\mu$-length of $P$ the \emph{length} of $P$, denoted by $\tau(P)=l$. Let $\mcalC_G(l):=\{C \in  \mcalC_G \mid \tau(C)=l\}$, $\mcalC_G(\alpha):=\{C \in  \mcalC_G \mid  C^{\ab}=\alpha \}$ and $\mcalC_G(\alpha,l):=\{C \in  \mcalC_G \mid \tau(C)=l,\ C^{\ab}=\alpha \}$. In other words, $\mcalC_G(l)$ is the 
set of all circuits of length $l$ on $G$, $\mcalC_G(\alpha)$ is the set of circuits $C$ on $G$ such that the abelianization of $C$ is $\alpha$,  and $\mcalC_G(\alpha,l)$ is the set of all circuits $C$ of length $l$ on $G$ such that the abelianization of $C$ is $\alpha$.

\begin{problem}[\textbf{Homological Circuit Detection Problems}] \label{Pr:HCDP}
 Consider a circulation $\alpha\in H_1(G,\mbbZ)$ on a connected graph $G$. We have the following \emph{Homological Circuit Detection Problems} with respect to $\alpha$:
\begin{enumerate}[(i)]
\item  Find a circuit in $\mcalC_G(\alpha)$;
\item Find a circuit in $\mcalC_G(\alpha, l)$ for a given length $l$;
\item Find a circuit in $\mcalC_G(\alpha)$ of minimum length. Or equivalently, determine the smallest $l$ such that  $\mcalC_G(\alpha,l)$ is non-empty and find a circuit in  $\mcalC_G(\alpha,l)$. 
\item For a metric $\mu$ of $G$, find a circuit in $\mcalC_G(\alpha)$ of minimum $\mu$-length. (Note that (iii) generalizes to (iv), which we call the \emph{Homological Shortest Circuit Detection Problem}. ) 
\end{enumerate}
\end{problem}

In this paper, we focus on the Homological Shortest Circuit Detection Problem (HSCDP) for the case $\alpha\neq 0$. Our solutions are related to a notion of direction-consistent circuits/cycles, as defined in below. 

\begin{definition}
For an orientation $\mfrako$ of $G$, we say a walk P=$\va_1\cdots\va_L$ is \emph{$\mathfrak{o}$-admissible} if it is a walk on the digraph $G_\mathfrak{o}$, or equivalently $\va_i\in \vE_\mathfrak{o}(G)$ for $i=1,\cdots,L$.  A walk is called a \emph{direction-consistent walk} if it is $\mathfrak{o}$-admissible with respect to a certain orientation $\mathfrak{o}$. If a circuit $C$ is direction-consistent, we say $C$ is a \emph{direction-consistent circuit} and $[C]$ is a \emph{direction-consistent cycle}. 
\end{definition}

The paper is organized as follows: in Section~\ref{S:Prelim}, we define some basic notions; in Section~\ref{S:DCC}, we solve the problem of detecting and counting direction-consistent circuits, as generalization of such problems for Eulerian circuits on digraphs; in Section~\ref{S:HSCDP}, we give a full solution to HSCDP; in Section~\ref{S:1cTRP}, as an application, we propose the one-carrier Transportation Routing Problem and provide a solution based on the solution to HSCDP. 

\section{Basic notions} \label{S:Prelim}
We first introduce some related notions. 
\begin{definition}
A \emph{partial orientation} $\mfrako$ of $G$ is an orientation of a subgraph $H$ of $G$ with no isolated vertices. We say $H$ is the \emph{support} of $\mfrako$, i.e.,  $\supp(\mfrako):=H$.  (In this sense, $\mfrako$ is considered both as a partial orientation of $G$ and an orientation of $\supp(\mfrako)$.) We say $\mfrako$ is \emph{connected} if $\supp(\mfrako)$ is connected. 
\end{definition}

\begin{definition} \label{D:lift}
Consider $\alpha\in C_1(G,\mbbZ)$ and suppose $\alpha\neq 0$. 
\begin{enumerate}[(i)]
\item $\alpha$ induces uniquely a partial orientation $\mfrako(\alpha)$ of $G$, i.e., $\alpha=\sum_{\ve\in \vE_{\mfrako(\alpha)}(H)}c_{\ve}\cdot\ve$ where $H=\supp(\mfrako(\alpha))$ with all the coefficients $c_{\ve}>0$. The \emph{support} of $\alpha$ is $\supp(\alpha):=\supp(\mfrako(\alpha))$.  We also say $\alpha$ is \emph{connected} if $\supp(\alpha)$ is connected, and $\alpha$ is \emph{universal}  on $G$ if $\supp(\alpha)=G$. (Clearly $\alpha$ is universal  is equivalent to $\mfrako(\alpha)$ is an orientation on $G$.)
\item For $\alpha=\sum_{\ve\in \vE_{\mfrako(\alpha)}(\supp(\alpha))}c_{\ve}\cdot\ve$, we call a graph $\widetilde{G(\alpha)}$  the \emph{lift} of $G$ corresponding to $\alpha$, which is formed by replacing each edge $e$ in $E(\supp(\alpha))$ by $c_\ve$ parallel edges where $e$ is the underlying edge of $\ve$, and call a partial orientation $\widetilde{\mfrako(\alpha)}$ on $\widetilde{G(\alpha)}$ the \emph{lift} of  $\mfrako(\alpha)$ where all aforementioned $c_\ve$ parallel edges in $\widetilde{G(\alpha)}$ are oriented the same way as $\ve$ in $G$ for all $\ve\in \vE_{\mfrako(\alpha)}(\supp(\alpha))$. 
\item The out-degree of $\alpha$ at $v\in V(G)$ is $\deg_\alpha^+(v):=\sum_{\ve\in \vE_{\mfrako(\alpha)}(\supp(\alpha)), ,\ve(0)=v}c_{\ve}$, and the in-degree of $\alpha$ at $v$ is $\deg_\alpha^-(v):=\sum_{\ve\in \vE_{\mfrako(\alpha)}(\supp(\alpha)), ,\ve(1)=v}c_{\ve}$. 
\item  If $\alpha$ is universal on $G$, then we call $\alpha$ is \emph{a (nonwhere-zero) flow} on $G$ if in addition $\alpha\in H_1(G,\mbbZ)$.
\end{enumerate}
\end{definition}

\begin{definition} \label{D:L1}
For a metric $\mu$ of $G$, we may define an $L^1$ norm on $C_1(G,\mbbZ)$ as $\Vert \alpha \Vert_\mu = \sum_{\ve\in \vE_{\mfrako}(G)}|c_\ve|\tau_\mu(\ve) $ for $\alpha = \sum_{\ve\in \vE_{\mfrako}(G)}c_\ve\cdot \ve$. In particular, if $\mu$ is trivial, then we simply write $\Vert\cdot \Vert_\mu$ as $\Vert\cdot\Vert$.
\end{definition}

The following lemma can be easily verified by definition. 

\begin{lemma} \label{L:ineq}
 $\tau_\mu(C)\geq \Vert C^{\ab}\Vert_\mu$ for all $C\in \mcalC_G$. 
\end{lemma}

\section{Detecting and counting direction-consistent circuits for connected circulations} \label{S:DCC}
The notion of direction-consistent circuits on a graph can be considered as a generalization of the notion of Eulerian circuits on an Eulerian digraph. In this section, we will discuss the detecting and counting problems of direction-consistent circuits for connected circulations $\alpha\in H_1(G,\mbbZ)$, with results generalizing the Hierholzer's algorithm for detecting Eulerian circuits and the BEST theorem for counting Eulerian circuits. 

\subsection{Detecting direction-consistent circuits: generalized Hierholzer's algorithm}
Hierholzer's Algorithm~\cite{BLW1998graph} is a classic method for finding Eulerian circuits on an Eulerian graph $G$, proposed by Carl Hierholzer in 1873. Hierholzer's Algorithm   is highly efficient, taking linear time $O(|E|)$ where $|E|$ is the number of edges on $G$.  It constructs paths incrementally until an Eulerian circuit is formed, and can be described as follows:  
\begin{enumerate}
\item Choose a starting point: Start from any vertex in $G$.
\item Construct an initial path: Traverse unvisited edges from the current vertex, marking them as visited, until you return to the starting vertex, forming a closed path.
\item Handle remaining edges: Check vertices in the current path. If a vertex has unvisited edges, start a new closed path from that vertex and embed it into the existing path.
\item Complete the path: Repeat the process until all edges are visited, resulting in the complete Eulerian circuit on $G$. 
\end{enumerate}

The following lemma can be straightforwardly verified. 
\begin{lemma} \label{L:dcc}
For each nonzero $\alpha\in H_1(G,\mbbZ)$, $\mcalC_G(\alpha,l)$ is empty for all $l<\Vert\alpha\Vert$, and if $\mcalC_G(\alpha,\Vert\alpha\Vert)$ is nonempty, then $\alpha$ is connected and all circuits in $\mcalC_G(\alpha,\Vert\alpha\Vert)$ are direction-consistent.
\end{lemma}

We will show in Theorem~\ref{T:DCC} that the converse statement is also true, i.e., as long as a nonzero  $\alpha\in H_1(G,\mbbZ)$ is connected, we may always find a direction-consistent circuit whose abelianization is $\alpha$. To prove this theorem, we may apply  an approach analogous to Hierholzer's algorithm to construct a desirable direction-consistent circuit. The procedures are shown in Algorithm~\ref{A:DCC}.

We will need the following lemma in our proof of Theorem~\ref{T:DCC}, which can be easily verified. 
\begin{lemma} \label{L:balancing}
For $C\in \mcalC_G$, we must have $C^{\ab}\in H_1(G,\mbbZ)$. For a non-closed path $P$, if the balancing condition (conf. Definition~\ref{D:homology}) is not satisfied for $P^{\ab}$ at vertex $v$, then $v$ is either $P(0)$ or $P(1)$. Specifically, we have  $\deg_{P^{\ab}}(P(0))=1$ and $\deg_{P^{\ab}}(P(1))=-1$. 
\end{lemma}

\begin{theorem} \label{T:DCC}
For a nonzero $\alpha\in H_1(G,\mbbZ)$, $\mcalC_G(\alpha,\Vert\alpha\Vert)$ is nonempty if and only if $\alpha$ is connected. 
\end{theorem}

\begin{proof}
It is clear that $\mcalC_G(\alpha,\Vert\alpha\Vert)$ being nonempty implies that $\alpha$ is connected, as stated in Lemma~\ref{L:dcc}. 

Now assuming that $\alpha$ is connected, we want to find a direction-consistent circuit in $\mcalC_G(\alpha,\Vert\alpha\Vert)$. Starting from any vertex $v\in V(\supp(\alpha))$, we can derive a direction consistent path $C$  by running Algorithm~1. It suffices to prove the claims in Algorithm~1. 
\begin{algorithm}
\label{A:DCC}
\caption{(Generalized Hierholzer's algorithm). Detecting direction-consistent circuits}
\begin{algorithmic}[1]
\State \textbf{Input:} A connected circulation  $\alpha\in H_1(G,\mbbZ)$   on a graph $G$
\State \textbf{Output:} A direction consistent circuit $C\in\mcalC_G(\alpha,\Vert\alpha\Vert)$
\State \textbf{Some notations:} 
\begin{enumerate}
\item $P(0)$ and $P(1)$ denote the initial and terminal vertices of a path $P$ respectively. 
\item A path concatenation notation will be used, namely, for $P_1 = \va_1\cdots \va_s$ and $P_2 = \vb_1\cdots \vb_t$, if $P_1(1)=P_2(0)$, then $\concat(P_1,P_2)=\va_1\cdots \va_s\vb_1\cdots \vb_t$. 
\item For a vertex $v$ and a partial orientation $\mfrako$, $\Out_\mfrako(v):=\{\ve\in\vE_\mfrako(\supp(\mfrako)) \mid \ve(0)=v\}$.
\item Translating a circuit $C=\va_0\cdots\va_{L-1}\in \mcalC(G)$ by $k\in\mbbZ$, one obtains $\mathfrak{t}_k(C)=\va_{t_0}\cdots \va_{t_{L-1}}$ where $t_i  =( i+k) \mod L$.
\end{enumerate}
\State \textbf{Initialization:} Choose an arbitrary vertex $v\in V(\supp(\alpha))$ and let $\mfrako=\mfrako(\alpha)$ and $\beta=\alpha$. Let $P$ be an empty path. 
\State \textbf{Procedure:}
\While{$\beta\neq 0$}
\While{$\Out_\mfrako(v)\neq \emptyset$}
\State Choose arbitrarily an oriented edge $\ve\in \Out_\mfrako(v)$
 \State $P\gets \concat(P,\ve)$ 
 \State $v\gets \ve(1)$
 \State $\beta \gets \beta -\ve$
 \State $\mfrako\gets \mfrako(\beta)$
 \EndWhile
 \State $k\gets 0$ \Comment{\textbf{Claim 1:} At this point, $P$ is a circuit.}
 \While{$k< \tau(P)$ and $\Out_\mfrako(v)= \emptyset$}
 \State $k\gets k+1$
 \State $P\gets \mathfrak{t}_1(P)$
 \State $v\gets P(1)$
 \EndWhile
 \EndWhile \Comment{\textbf{Claim 2:} The whole loop will always be exited.} 
 \State $C=P$  \Comment{\textbf{Claim 3:} At this point, $C\in\mcalC_G(\alpha,\Vert\alpha\Vert)$.}  
\State \Return{$C$}  
 \end{algorithmic}
 \end{algorithm}

We first note that at Step~9 of Algorithm~\ref{A:DCC}, we concatenate path $P$ with an oriented edge $\ve$, and at Step~11, we deduct $\ve$ from $\beta$, this means the relation $P^{\ab}+\beta=\alpha$ is always maintained after consecutively executing these two operations. 

For Claim~1, suppose for contradiction that at Step~14, $P$ is not a circuit. Then $P(0)\neq P(1)$, and by Lemma~\ref{L:balancing}, $\deg_{P^{\ab}}(P(1))=-1$. This means $\deg_\beta(P(1))=1$ since $\beta=\alpha-P^{\ab}$ and $\alpha$ is a circulation. Hence $\Out_{\mfrako(\beta)}(P(1))$ must be nonempty, a contradiction. 

Now assume that at Step~14, we derive a circuit $P$ and at this point $\beta = \alpha-P^{\ab}$. Claim that $\Out_{\mfrako(\beta)}(v)=\emptyset$ for all vertex $v$ on $P$ if and only if $\beta=0$. This is also clear, since for all $v\in V(\supp(P^{\ab}))\bigcap V(\supp(\beta))$, $\Out_{\mfrako(\beta)}(v)$ must be nonempty. This means the subgraphs $\supp(P^{\ab})$ and $\supp(\beta)$ must be disjoint. But $\alpha=P^{\ab}+\beta$ is connected as assumed. 

Therefore, the ``while'' loop at Step~20 will always be exited with $\beta = 0$  and $P^{\ab}=\alpha$ as claimed in Claim~2. Consequently, at  Step~21, a direction-consistent circuit $C\in\mcalC_G(\alpha,\Vert\alpha\Vert)$ will be returned as claimed in Claim~3. 

\end{proof}

\begin{example}\label{E:K4-detect}
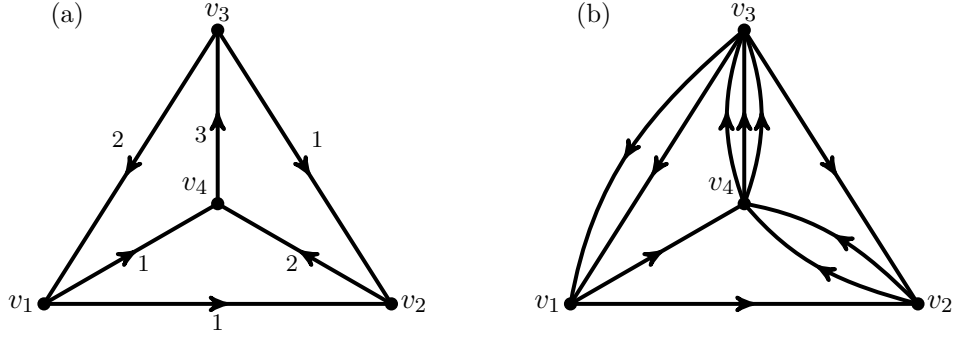
\begin{figure}
 \centering
\begin{tikzpicture}[>=to,x=2cm,y=2cm]

\tikzset{
        midarrow/.style={
            postaction={decorate},
            decoration={
                markings,
                mark=at position 0.53 with {\arrow[line width=1.5pt]{stealth'}}
            }
        }
    }

\begin{scope}[shift={(0,0)}] 
\draw (-1, 1.25) node {(a)};

\coordinate (v1) at ($2/3*(-1.732,-1)$);
\coordinate (v2) at  ($2/3*(1.732,-1)$);
\coordinate (v3) at ($2/3*(0,1.732)$);
\coordinate (v4) at (0,0);

\path[-,font=\scriptsize,  line width=1.5pt, black]
(v1) edge[midarrow] node[pos=0.5,anchor=north,font=\small]{1} (v2)
(v1) edge[midarrow] node[pos=0.4,right=5pt,font=\small]{1} (v4)
(v2) edge[midarrow] node[pos=0.4,left=5pt,font=\small]{2} (v4)
(v3) edge[midarrow] node[pos=0.4,left=5pt,font=\small]{2} (v1)
(v3) edge[midarrow] node[pos=0.4,right=5pt,font=\small]{1} (v2)
(v4) edge[midarrow] node[pos=0.4,left,font=\small]{3} (v3);

\fill [black] (v1) circle (2.5pt);
\draw (v1) node[anchor=east] {\large $v_1$};
\fill [black] (v2) circle (2.5pt);
\draw (v2) node[anchor=west] {\large $v_2$};
\fill [black] (v3) circle (2.5pt);
\draw (v3) node[anchor=south] {\large $v_3$};
\fill [black] (v4) circle (2.5pt);
\draw (v4) node[anchor=south east] {\large $v_4$};
\end{scope}

\begin{scope}[shift={(3.5,0)}] 
\draw (-1, 1.25) node {(b)};

\coordinate (v1) at ($2/3*(-1.732,-1)$);
\coordinate (v2) at  ($2/3*(1.732,-1)$);
\coordinate (v3) at ($2/3*(0,1.732)$);
\coordinate (v4) at (0,0);

\path[-,font=\scriptsize,  line width=1.5pt, black]
(v1) edge[midarrow] (v2)
(v1) edge[midarrow]  (v4)
(v2) edge[midarrow,out=135,in=-10] (v4)
(v2) edge[midarrow,out=165,in=-50] (v4)
(v3) edge[midarrow] (v1)
(v3) edge[midarrow,out=-140,in=80]  (v1)
(v3) edge[midarrow] (v2)
(v4) edge[midarrow]  (v3)
(v4) edge[midarrow,out=70,in=-70] (v3)
(v4) edge[midarrow,out=110,in=-110] (v3);

\fill [black] (v1) circle (2.5pt);
\draw (v1) node[anchor=east] {\large $v_1$};
\fill [black] (v2) circle (2.5pt);
\draw (v2) node[anchor=west] {\large $v_2$};
\fill [black] (v3) circle (2.5pt);
\draw (v3) node[anchor=south] {\large $v_3$};
\fill [black] (v4) circle (2.5pt);
\draw (v4) node[anchor=south east] {\large $v_4$};
\end{scope}
\end{tikzpicture}
 \caption{An example for $K_4$.}
 \label{F:K4}
\end{figure}
Let $G$ be the complete graph $K_4$ with an orientation illustrated in Figure~\ref{F:K4}(a). The vertex set if $V(G)=\{v_1,v_2,v_3,v_4\}$, and we label the oriented edges in $\vE_\mfrako(G)$ from $v_i$ to $v_j$ as $\ve_{ij}$. Consider a connected circulation $\alpha=\ve_{12}+\ve_{14}+2\ve_{24}+2\ve_{31}+\ve_{32}+3\ve_{43}$. Let us run Algorithm~1 for $\alpha$. We let the path $P$ start from $v_1$. Initially we may let $P$ follow either $\ve_{14}$ or $\ve_{12}$. Choose to follow $\ve_{14}$. Keep extending $P$ in this way, we may get $P=\ve_{14}\ve_{43}\ve_{31}\ve_{12}\ve_{24}\ve_{43}\ve_{31}$. At this stage $P$ is a circuit and the ``while'' loop is exited, and we have to translate $P$ to find a ``non-saturated'' vertex. The first such a vertex we get is $v_4$. By translation, now $P$ becomes $\ve_{43}\ve_{31}\ve_{12}\ve_{24}\ve_{43}\ve_{31}\ve_{14}$. Extending $P$, we will finally get a direction-consistent circuit $C=\ve_{43}\ve_{31}\ve_{12}\ve_{24}\ve_{43}\ve_{31}\ve_{14}\ve_{43}\ve_{32}\ve_{24}$ such that $C^{\ab}=\alpha$. 
\end{example}

\subsection{Counting direction-consistent circuits: generalized BEST theorem}
The well-known BEST theorem, named after de Bruijn and van Aardenne-Ehrenfest~\cite{AB1951circuits} and Smith and Tutte~\cite{ST1941unicursal}, provides an efficient algorithm of counting Eulerian cycles on an Eulerian digraph $G_\mfrako$ (meaning that the in-degree and out-degree at each vertex of $G_\mfrako$ are identical). 

Recall that an arborescence $T$ of $G_\mfrako$ with a root vertex $w$ is a spanning tree of $G$ such that all edges  on $T$ are oriented towards $w$ with respect to $\mfrako$. 
Recall also that the Laplacian matrix of $G_\mfrako$, denoted $L(G_\mfrako)$, is defined as: (i) for $i\neq j$, $(L(G_\mfrako))_{ij}$ is the negation of the number of positively oriented edges with respect to $\mfrako$ with initial vertex $v_i$ and terminal vertex $v_j$, and (ii)  $(L(G_\mfrako))_{ii}=-\sum_{ j\neq i}(L(G_\mfrako))_{ij}$. Removing the row and column of  $L(G_\mfrako)$ with respect to $w$, we get the reduced Laplacian $L^{\red(w)}(G_\mfrako)$ at $w$.

 Denote by $\ec(G,\mfrako)$ the number of Eulerian cycles on $G_\mfrako$, and $\In_w(G,\mfrako)$ the number of arborescences $G_\mfrako$ rooted at $w$. Then the BEST theorem states that $\ec(G,\mfrako)=\In_w(G,\mfrako)\cdot \prod_{v\in V(G)}(\deg^+_\mfrako(v)-1)!$ where $\deg^+_\mfrako(v)$ is the out-degree of $G_\mfrako$ at $v\in V(G)$. Note that by a directed version of Kirchhoff's matrix-tree theorem, $\In_w(G,\mfrako)$ is actually independent of $w$, and can be efficiently computed as the determinant of the reduced Laplacian matrix $L^{\red(w)}(G_\mfrako)$. 

Before making a generalization of  the BEST theorem, we first introduce a notion of weighted Laplacian here. 

\begin{definition} \label{D:laplacian}
Consider a graph $G$. Let $\alpha\in C_1(G,\mbbZ)$ be universal. Let $c_\ve = \deg_\alpha(\ve)$ for all $\ve\in\vE(G)$. Then $\alpha=\sum_{\ve\in \vE_{\mfrako(\alpha)}(G)}c_{\ve}\cdot\ve$ with all $c_{\ve}>0$.  The \emph{weighted Laplacian} $L_\alpha(G)$ of $G$ with respect to $\alpha$ is defined as: $(L_\alpha(G))_{ij}=-\sum_{\ve\in\vE_{\mfrako(\alpha)}(G),\ve(0)=v_i,\ve(1)=v_j}c_\ve$ for all $i\neq j$, and $(L_\alpha(G))_{ii}=\sum_{\ve\in\vE_{\mfrako(\alpha)}(G),\ve(0)=v_i,\ve(1)\neq v_i}c_\ve$ for all diagonal entries. Removing the row and column of  $L_\alpha$ with respect to $w$, we get the \emph{reduced weighted Laplacian} $L_\alpha^{\red(w)}(G)$.
\end{definition}

\begin{lemma} \label{L:laplacian}
Let $\alpha\in C_1(G,\mbbZ)$  be universal. Then $L_\alpha(G)=L\left(\widetilde{G(\alpha)}_{\widetilde{\mfrako(\alpha)}}\right)$, and for any $w\in V(G)$, $L_\alpha^{\red(w)}(G)=L^{\red(v)}\left(\widetilde{G(\alpha)}_{\widetilde{\mfrako(\alpha)}}\right)$. Moreover, $\alpha$ is a flow on $G$ (meaning that $\alpha$ is in $H_1(G,\mbbZ)$ and universal) if and only if $\widetilde{G(\alpha)}_{\widetilde{\mfrako(\alpha)}}$ is an Eulerian digraph. 
\end{lemma}
\begin{proof}
Easily verifiable by the definitions of the Laplacian and reduced Laplacian of a digraph,  the weighted Laplacian and reduced weighted Laplacian (Definition~\ref{D:laplacian}), and lifts of graphs and orientations (Definition~\ref{D:lift}(ii)). 
\end{proof}

The following theorem, which we call the generalized BEST theorem, provides a formula of counting the total number of direction-consistent circuits with a prescribed abelianization.
 
\begin{theorem}[Generalized BEST Theorem] \label{T:BEST}
For a flow $\alpha$ on $G$ and any $w\in V(G)$, we have  
\begin{align*}
|\mathcal{C}_G(\alpha, \Vert\alpha\Vert)|&=\Vert\alpha\Vert\cdot \ec\left(\widetilde{G(\alpha)},\widetilde{\mfrako(\alpha)}\right)\cdot \prod_{\ve\in\vE_{\mfrako(\alpha)}(G)}\left(\deg_\alpha(\ve)!\right)^{-1} \\
&=\Vert\alpha\Vert\cdot \det\left(L_\alpha^{\red(w)}(G)\right)\cdot \prod_{v\in V(G)}(\deg^+_\alpha-1)! \cdot  \prod_{\ve\in\vE_{\mfrako(\alpha)}(G)}\left(\deg_\alpha(\ve)!\right)^{-1}.
\end{align*}
\end{theorem}

\begin{proof}
First note that by Lemma~\ref{L:dcc}, $|\mathcal{C}_G(\alpha, \Vert\alpha\Vert)|$ is the number of direction-consistent circuits whose abelianization is $\alpha$. 

By Definition~\ref{D:lift}(ii) and Lemma~\ref{L:laplacian}, $\widetilde{G(\alpha)}_{\widetilde{\mfrako(\alpha)}}$ must be an Eulerian digraph, and there is a natural correspondence between Eulerian circuits on $\widetilde{G(\alpha)}_{\widetilde{\mfrako(\alpha)}}$  and direction-consistent circuits in $\mathcal{C}_G(\alpha, \Vert\alpha\Vert)$. In particular, every Eulerian circuit on $\widetilde{G(\alpha)}_{\widetilde{\mfrako(\alpha)}}$ corresponds to a direction-consistent circuit in $\mathcal{C}_G(\alpha, \Vert\alpha\Vert)$. Under this correspondence, one can straightforwardly verify that for each $C\in \mathcal{C}_G(\alpha, \Vert\alpha\Vert)$, there are exactly  $\prod_{\ve\in\vE_{\mfrako(\alpha)}(G)}\left(\deg_\alpha(\ve)!\right)$ Eulerian circuits on $\widetilde{G(\alpha)}_{\widetilde{\mfrako(\alpha)}}$  corresponding to the direction-consistent circuit$C$ on $G$. Note also that the ratio of the number of Eulerian circuits and the number of Eulerian cycle is exactly the number of edges. Hence the first identity in the statement follows. The second identity follows from applying the formula of the BEST theorem to $\widetilde{G(\alpha)}_{\widetilde{\mfrako(\alpha)}}$ and identifying $L(\widetilde{G(\alpha)}_{\widetilde{\mfrako(\alpha)}})$ with $L_\alpha(G)$ (Lemma~\ref{L:laplacian}). 
\end{proof}

\begin{remark}
If $\alpha\in H_1(G,\mbbZ)$ is connected but not universal on $G$, we just need to apply the above generalized BEST theorem to $\supp(\alpha)$ to compute the number of direction-consistent circuits with prescribed abelianization $\alpha$. 
\end{remark}

\begin{example}
In Example~\ref{E:K4-detect}, we investigate a $4$-complete graph $G$, and show how to detect a direction-consistent circuit whose abelianization is $\alpha=\ve_{12}+\ve_{14}+2\ve_{24}+2\ve_{31}+\ve_{32}+3\ve_{43}$. Now let us count the number of all such direction-consistent circuits. Figure~\ref{F:K4}(b) shows the lift $\widetilde{G(\alpha)}$ of $G$ and the lift $\widetilde{\mfrako(\alpha)}$ of $\mfrako(\alpha)$. Then by computation, $\Vert\alpha\Vert=10$, $\ec\left(\widetilde{G(\alpha)},\widetilde{\mfrako(\alpha)}\right)=48$ and $\prod_{\ve\in\vE_{\mfrako(\alpha)}(G)}\left(\deg_\alpha(\ve)!\right)=24$. Therefore, by Theorem~\ref{T:BEST}, we have $|\mathcal{C}_G(\alpha, \Vert\alpha\Vert)|=20$. 
\end{example}

\section{Solving the Homological Shortest Circuit Detection Problem} \label{S:HSCDP}
Algorithm~1 solves the HSCDP (Problem~\ref{Pr:HCDP}(iii)(iv)) when the targeted circulation $\alpha\in H_1(G,\mbbZ)$ is connected. 

To solve the HSCDP in general, we need to consider the case when $\supp(\alpha)$ have several connected components.  Recall that we have the following notion of subgraph contraction. 

\begin{definition}
For a graph $G$ and a subgraph $H$ of $G$, the \emph{subgraph contraction}  of $G$ with respect to $H$, denoted by $G|_H$, is a graph derived by contracting $G$ on all edges in $E(H)$. There is a natural map $\gamma_H$ from $G$  to  $G|_H$, which sends each connected component of $H$ to a single vertex of $G|_H$, while leaving the remaining vertices and edges untouched. 
\end{definition}

\begin{example}
\begin{figure}
 \centering
\begin{tikzpicture}[>=to,x=.8cm,y=.8cm]
    \tikzset{
        midarrow/.style={
            postaction={decorate},
            decoration={
                markings,
                mark=at position 0.5 with {\arrow{>}}
            }
        }
    }

    \begin{scope}[shift={(0,0)}] 
    \draw (-1, 1) node {(a)};

    \coordinate (V1) at (5, 0);
    \coordinate (V2) at (5, {-sqrt(3)-2});
    \coordinate (V3) at (5, {-sqrt(3)-4});
    \coordinate (V4) at (2, {-sqrt(3)-4});
    \coordinate (V5) at (0, {-sqrt(3)-2});
    \coordinate (V6) at (2, {-sqrt(3)-2});
    \coordinate (V7) at (0, {-sqrt(3)});
    \coordinate (V8) at (2, {-sqrt(3)});
    \coordinate (V9) at (1, 0);

    \path[-,font=\scriptsize, line width=1.5pt, black]
        (V1) edge (V2)
        (V2) edge[midarrow, out=-35, in=35, blue]  node[midway, right, font=\large, blue] {2} (V3)
        (V3) edge[midarrow, out=145, in=-145, blue]  node[midway, left, font=\large, blue] {2} (V2)
        (V3) edge (V4)
        (V2) edge (V6)
        (V6) edge (V1)
        (V5) edge[midarrow, blue] node[midway, above, font=\large, blue] {1}  (V6)
        (V6) edge[midarrow, blue] node[midway, right, font=\large, blue] {1}  (V4)
        (V4) edge[midarrow, blue, out=180, in=-90]  node[midway, left=5pt, font=\large, blue] {2} (V5)
        (V5) edge[midarrow, blue, out=-15, in=105]  node[midway, right, font=\large, blue] {1} (V4)
        (V5) edge (V7)
        (V6) edge (V8)
        (V8) edge[midarrow, blue] node[midway, below, font=\large, blue] {1}  (V7)
        (V7) edge[midarrow, blue] node[midway, above, left, font=\large, blue] {1}  (V9)
        (V9) edge[midarrow, blue] node[midway, above, right, font=\large, blue] {1}  (V8)
        (V9) edge (V1);
        
    \node[above] at (V1) {$v_1$};
    \node[right=10pt, blue] at (V2) {$v_2$};
    \node[below, blue] at (V3) {$v_3$};
    \node[below, blue] at (V4) {$v_4$};
    \node[below left, blue] at (V5) {$v_5$};
    \node[above left, blue] at (V6) {$v_6$};
    \node[above left, blue] at (V7) {$v_7$};
    \node[right, blue] at (V8) {$v_8$};
    \node[above, blue] at (V9) {$v_9$};
    \fill[black] (V1) circle (2.5pt);
    \fill[blue] (V2) circle (2.5pt);
    \fill[blue] (V3) circle (2.5pt);
    \fill[blue] (V4) circle (2.5pt);
    \fill[blue] (V5) circle (2.5pt);
    \fill[blue] (V6) circle (2.5pt);
    \fill[blue] (V7) circle (2.5pt);
    \fill[blue] (V8) circle (2.5pt);
    \fill[blue] (V9) circle (2.5pt);
    \end{scope}

    \begin{scope}[shift={(9,0)}] 
    \draw (-1, 1) node {(b)};
    \coordinate (V1) at (4, 0);
    \coordinate (W2) at (0, -4);
    \coordinate (W1) at (4, -4);
    \coordinate (W3) at (0, 0);

    \path[-,font=\scriptsize, line width=1.5pt, black]
        (W3) edge (V1)
        (V1) edge (W1)
        (V1) edge (W2)
        (W1) edge[out=155, in=25] (W2)
        (W2) edge[out=-25, in=-155] (W1)
        (W2) edge[out=65, in=-65] (W3)
        (W3) edge[out=-115, in=115] (W2);

    \node[below, right, blue] at (W1) {$w_1$};
    \node[below, left, blue] at (W2) {$w_2$};
    \node[above, blue] at (W3) {$w_3$};
    \node[above] at (V1) {$v_1$};
    \fill[blue] (W1) circle (2.5pt);
    \fill[blue] (W2) circle (2.5pt);
    \fill[blue] (W3) circle (2.5pt);
    \fill[black] (V1) circle (2.5pt);
    \end{scope}
    
\end{tikzpicture}
\caption{(a) A graph $G$ with a subgraph $H$ (in blue);  (b) The subgraph contraction $G|_H$ of $G$ with respect to $H$.}
 \label{F:Contraction}
\end{figure}
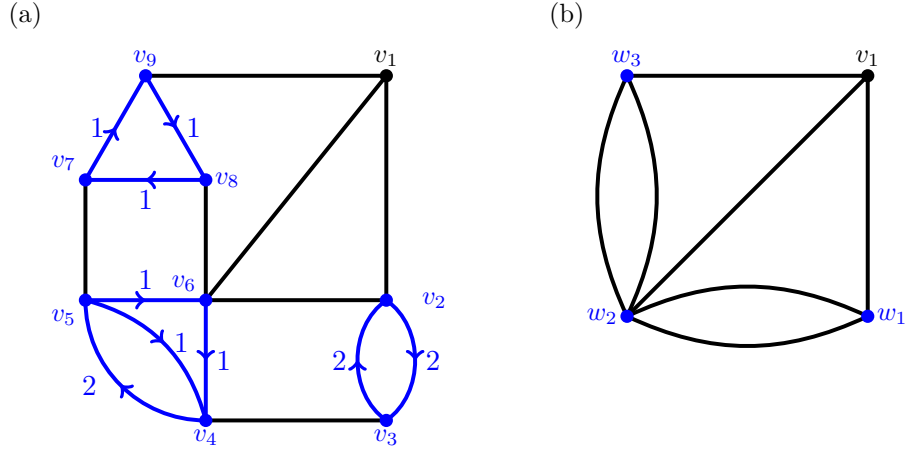
Figure~\ref{F:Contraction} shows an example of subgraph contraction. Figure~\ref{F:Contraction}(a) shows a graph $G$ with vertex set $\{v_1,\cdots, v_9\}$. The blue vertices and edges form a subgraph $H$ with three connected components $H_1$, $H_2$ and $H_3$. In particular, $V(H_1)=\{v_2,v_3\}$, $V(H_2)=\{v_4,v_5,v_6\}$ and $V(H_3)=\{v_7,v_8,v_9\}$. Contracting $G$ with respect to $H$, one obtain $G|_H$ as shown in Figure\ref{F:Contraction}(b). Note that $\gamma_H(H_1)=w_1$,  $\gamma_H(H_2)=w_2$ and  $\gamma_H(H_3)=w_3$.
\end{example}

Let us recall the Steiner Tree Problem in graphs \cite{Ljubic21solving,DW1971steiner}  stated as follows.

\begin{problem}[Steiner Tree Problem in graphs]
Consider a graph $G$ with non-negative edge weights and a subset $V'$ of $V(G)$. A \emph{Steiner Tree} is a subtree $T$ of $G$ such that  $V'\subset V(T)$. The task of the \emph{Steiner Tree Problem} is to find a Steiner tree of minimum weight. Such an optimization problem is NP-hard problem, while some special cases can be solved efficiently. 
\end{problem}

Now we can present our solution to HSCDP as in Algorithm~2. 
\begin{algorithm}
\label{A:HSCDP}
\caption{A solution to HSCDP (Problem~\ref{Pr:HCDP}(iv))}
\begin{algorithmic}[1]
\State \textbf{Input:} A nonzero circulation  $\alpha\in H_1(G,\mbbZ)$   on a graph $G$, and a metric $\mu$ on $G$
\State \textbf{Output:} A circuit in $\mcalC(\alpha)$ of minimum $\mu$-length
\State \textbf{Procedure:}
\If{$\alpha$ is connected}
\State Run Algorithm~1 for $\alpha$ to derive a circuit $C\in\mcalC_G(\alpha,\Vert\alpha\Vert)$
\Else
\State $H\gets \supp(\alpha)$
\State Obtain all the connected components $H_1,\cdots,H_s$ of $H$
\For{$i=1,\cdots,s$}
\State $\alpha_i\gets\alpha|_{H_i}$ \Comment{$\alpha|_{H_i}$ is the restriction of $\alpha$ to $H_i$}
\State Run Algorithm~1 for $\alpha_i$ to derive a circuit $C_i$
\EndFor
\State Contract $G$ with respect to $H$ and obtain $G|_H$
\State Obtain $W = \{\gamma_H(H_1),\cdots,\gamma_H(H_s)\}\subseteq V(G|_H)$ 
\State Solve the Steiner Tree Problem on $G|_H$: find a Steiner tree $T_\alpha$ of minimum $\mu$-length $\mu(T_\alpha)=\sum_{e\in E(T)}\mu(e)$ with respect to $W$ 
\State Obtain a closed path $P$ on $T_\alpha$ with initial vertex $w_1$ such that $P$ goes through each  edge of $T_\alpha$ exactly twice 
\State Partition $P$ into paths $P_1,\cdots,P_t$ such that the initial and terminal vertices of each $P_i$ are in W but no other interior vertices in $W$
\State Identify $P_1,\cdots,P_t$ as paths on $G$
\For{$i=1,\cdots,s$}
\State Obtain $V_i:=V(H_i)\bigcap\left( \bigcup_{i=1}^{s}V(P_i)\right)$ 
\State Translate $C_i$ if necessary to make $C_i$ start at a vertex in $V_i$
\State Partition $C_i$ into paths $P_{i,1},\cdots,P_{i,|V_i|}$ such that $C_i=\concat(P_{i,1},\cdots,P_{i,|V_i|})$ and the initial and terminal vertices of each $P_{i,j}$ are in $V_i$
\EndFor
\State Ordering $P_1,\cdots,P_t,P_{1,1},\cdots,P_{1,|V_1|},\cdots,P_{s,1},\cdots,P_{s,|V_s|}$ properly, concatenate these paths into a circuit $C$
\EndIf 
\State\Return C
 \end{algorithmic}
 \end{algorithm}
 
 \begin{theorem} \label{T:HSCDP}
 Algorithm~2 solves HSCDP. In particular, $\tau_\mu(C)=2\mu(T_\alpha)+\Vert\alpha\Vert_\mu$ for the derived circuit $C$ in Algorithm~2. 
 \end{theorem}
 \begin{proof}
First we claim that $V_i$ at Step~24 of Algorithm~2, by some proper ordering,  we can always concatenate $P_1,\cdots,P_t,P_{1,1},\cdots,P_{1,|V_1|},\cdots,P_{s,1},\cdots,P_{s,|V_s|}$ into a closed path $C$. Note that at Step~20, the derived $V_i$ can be the set of contact vertices between $H_i$ and $T$. In a concatenation for $C$, when $C$ travels along some $P_{i,j}$ and hits  $v=P_{i,j}(1)=P_{i,j'}(0)$, $C$ will not keep traveling along $P_{i,j'}$ as $C_i$. Instead, $C$ will exit $H_i$ and travel along the unique $P_k$ with 
$P_k(0)=v$. At a later stage, $C$ will travel back along the unique $P_{k'}$ with  $P_{k'}(1)=v$ and hit $v$, and then $C$ will travel along  $P_{i,j'}$. In this sense, $C$ will always be able to get back to its initial vertex, and at that stage, $C$ has traveled along all $P_1,\cdots,P_t,P_{1,1},\cdots,P_{1,|V_1|},\cdots,P_{s,1},\cdots,P_{s,|V_s|}$, each path for exactly one time. Therefore, we will finally get a closed path $C$ with $\tau_\mu(C)=\sum_{i=1}^t \tau_\mu(P_i)+\sum_{i=1}^s\sum_{j=1}^{|V_i|}\tau_\mu(P_{i,j})=2\mu(T_\alpha)+\sum_{i=1}^s\tau_\mu(C_i)$. Since all $C_i$'s are direction-consistent circuits, we must have $\tau_\mu(C)=2\mu(T_\alpha)+\sum_{i=1}^s\Vert\alpha_i\Vert_\mu=2\mu(T)+\Vert\alpha\Vert_\mu$ . 

Now let us show that $C\in\mcalC_G(\alpha)$. Actually $$C^{\ab}=\sum_{i=1}^t P_i^{\ab}+\sum_{i=1}^s\sum_{j=1}^{|V_i|}P_{i,j}^{\ab}=\sum_{i=1}^t P_i^{\ab}+\sum_{i=1}^s C_i^{\ab}.$$ Since $C_i$'s are derived by Algorithm~1 for $\alpha_i$, we have $C_i^{\ab}=\alpha_i$. Moreover, we must have $\sum_{i=1}^t P_i^{\ab}=0$. Therefore, $C^{\ab}=\alpha$ as claimed. 

It remains to show that $C$ has the minimum $\mu$-length among all circuits in $\mcalC_G(\alpha)$. Let $C'\in \mcalC_G(\alpha)$. For $i=1,\cdots,s$, let $P'_{i,1},\cdots, P'_{i,t_i}$ be all the maximal paths which are subpaths of $C'$ and are entirely contained in $H_i$. Let $P'_1,\cdots P'_{t'}$ be all the remaining paths when removing all $P'_{i,1},\cdots, P'_{i,t_i}$ for $i=1,\cdots,s$. Then it is clear that (1) $\sum_{j=1}^{t_i}{P'_{i,j}}^{\ab}=\alpha_i$ for $i=1,\cdots,s$, and (2) $\sum_{i=1}^{t'} {P'_i}^{\ab}=0$. But (1) implies that $\sum_{j=1}^{t_i}\tau_\mu(P'_{i,j})\geq \Vert\alpha_i\Vert_\mu$ and (2) implies that $\sum_{i=1}^{t'} \tau_\mu(P'_i)\geq 2\mu(T_\alpha)$. 
 \end{proof}
 \begin{example}
Let us consider $G$ and $G|_H$ again as shown in  Figure~\ref{F:Contraction}. An orientation $\mfrako$ of $H$ is demonstrated by the arrows. 
We label the oriented edges in $\vE_\mfrako(H)$ from $v_i$ to $v_j$ as $\ve_{ij}$. Consider a circulation $\alpha = 2\ve_{23}+2\ve_{32}+2\ve_{45}+\ve_{54}+\ve_{56}+\ve_{64}+\ve_{79}+\ve_{98}+\ve_{87}\in H_1(G,\mbbZ)$. Clearly $\supp(\alpha)=H=H_1\sqcup H_2 \sqcup H_3$. Running Algorithm~2 for $\alpha$, if the tree $T_\alpha$ derived at Step~15 is given by $V(T_\alpha)=\{v_1,w_1,w_2,w_3\}$ an $E(T_\alpha)=\{e'_{11},e'_{12},e'_{13}\}$ where $e'_{1i}$ is the edge connecting $v_1$ and $w_i$ for $i=1,2,3$, then an output of Algorithm~2 can be $C=\ve_{23}\ve_{32}\ve_{23}\ve_{32}\ve_{21}\ve_{16}\ve_{64}\ve_{45}\ve_{54}\ve_{45}\ve_{56}\ve_{61}\ve_{19}\ve_{98}\ve_{87}\ve_{79}\ve_{91}\ve_{12}$. If the the tree $T$ derived at Step~15 is made of the edge in $G|_H$ between $w_1$ and $v_2$ corresponding to the edge $e_{34}\in E(G)$ and the edge in $G|_H$ between $w_2$ and $w_3$ corresponding to the edge $e_{57}\in E(G)$, then an output of Algorithm~2 can be $C=\ve_{32}\ve_{23}\ve_{32}\ve_{23}\ve_{34}\ve_{45}\ve_{54}\ve_{45}\ve_{57}\ve_{79}\ve_{98}\ve_{87}\ve_{75}\ve_{56}\ve_{64}\ve_{43}$.
 \end{example}

\section{An application to  the one-carrier Transportation Routing Problem of type (1,1)} \label{S:1cTRP}
The Transportation Routing Problem~\cite{TGF74transportation} is a specific type of optimization problem that focuses on determining the most efficient routes for vehicles to transport goods from a single or multiple sources to multiple destinations. The transportation routing problem is a key component of logistics and distribution planning, as it helps companies optimize their delivery operations, reduce transportation costs, and improve overall efficiency. Based on variants of assumptions of constraints, Transportation Routing Problem can be related to some classical optimization problems,  such as the Traveling Salesman Problem~\cite{HPR2013traveling,JM1997traveling} and the Vehicle Routing Problem~\cite{PGGM2013a, BRN2016the}.

 Here we propose a special case of the Transportation Routing Problem, namely 1c-TRP of type (1,1). 

\begin{problem}[\textbf{One-carrier Transportation Routing Problem}] \label{Pr:1c-TRP}
Let $G$ be a connected graph with each edge  assigned a non-negative length. A set of items needs to be transported from their \emph{source vertices} to their \emph{destination vertices}. More precisely, each item is initially located at its source vertex and must be transported along the edges of $G$ to its destination vertex. There is exactly one carrier with a specified \emph{capacity} for the whole transportation task. The carrier starts at a specified \emph{starting vertex}, follows a planned route to pick up items at their source vertices, delivers them to their designated destination vertices, and finally returns to the starting vertex. During the whole transportation process, the number of items on the carrier \emph{can not exceed its capacity}. The objective is to \emph{minimize the total traveling distance} in order to optimize delivery costs. This is called the \emph{one-carrier Transportation Routing Problem}, or \emph{1c-TRP}. If further assuming that the source and destination vertices for each item are always adjacent, and the capacity for the carrier is $1$, then we call such a special case \emph{1c-TRP of type (1,1)}. 
\end{problem}

Our solution to this one-carrier Transportation Routing Problem of type (1,1) is based on relating it to the HSCDP.

Here let us first define some partial orderings. 
\begin{definition} \label{D:parital}
We endow the space of partial orientations of $G$ and $C_1(G,\mbbZ)$ with partial orderings as follows:
\begin{enumerate}[(i)]
\item For two partial orientations $\mfrako$ and $\mfrako'$ of $G$, we say $\mfrako\leq \mfrako'$ if $\supp(\mfrako)$ is a subgraph of $\supp(\mfrako')$, and $\mfrako'$ coincides with $\mfrako$  when restricted to $\supp(\mfrako)$. 
\item For $\alpha,\beta\in C_1(G,\mbbZ)$ and a partial orientation $\mfrako$, we say  $\alpha\leq_\mfrako \beta$ if  $\supp(\beta-\alpha)$ is a subgraph of $\supp(\mfrako)$ and $\beta-\alpha$ can be expressed as $\beta-\alpha=\sum_{\ve\in \vE_\mfrako(\supp(\mfrako))}c_\ve\cdot \ve$ with $c_\ve\geq 0$ for all $\ve\in \vE_\mfrako(\supp(\mfrako))$.  
\item For $\alpha,\beta\in C_1(G,\mbbZ)$, we say $\alpha\leq \beta$, if $\mfrako(\alpha)\leq \mfrako(\beta)$ and $\alpha\leq_{\mfrako(\beta)} \beta$. 
\end{enumerate}
\end{definition}

It is straightforward to verify that the above partial orderings are well-defined. For $\alpha\in C_1(G,\mbbZ)$, we let $\Lambda_G(\alpha):=\{\beta\in C_1(G,\mbbZ)\mid \alpha\leq \beta\}$. 

We can restate Problem~\ref{Pr:1c-TRP} based on the terminologies introduced above as follows:
\begin{problem}[A restatement of 1c-TRP of type (1,1)] \label{Pr:1c-TRP-new} 
Let $G$ be a simple graph with vertex set $V(G)=\{v_1,\cdots,v_n\}$.   Let $Q=(q_{ij})_{n\times n}$ be a \emph{task matrix}, i.e., for all $1\leq i,j\leq n$, the entry $q_{ij}\geq 0$ is the number of items that are initially located at $v_i$ and designated to be transported to $v_j$. Assume that $q_{ij}=0$ if $v_i$ and $v_j$ are not adjacent. Let $\mu$ be a metric on $G$. For a path $P$ and any two adjacent vertices $v_i$ and $v_j$, denote by $\theta_{ij}(P)$ the number of times that the path $P$ goes through the oriented edge from $v_i$ to $v_j$. The task is to find a closed walk $C$ starting from a predetermined vertex $v$, shortest with respect to $\mu$, such that $\theta_{ij}(C)\geq q_{ij}$ for all adjacent $v_i$ and $v_j$. 
\end{problem}

Note that the desired closed walk $C$ in Problem~\ref{Pr:1c-TRP-new} is not necessarily a circuit (backtracks and tail are allowed), and it might happen that $q_{ij}$ and $q_{ji}$ are both positive. To apply the results presented in Section~\ref{S:DCC}~and~\ref{S:HSCDP}, we will play the following ``edge doubling trick''. 

\begin{definition} \label{D:doubling}
Let $G$  be a simple graph with vertex set $V(G)=\{v_1,\cdots,v_n\}$ and edge set $E(G)=\{e_1,\cdots,e_m\}$.
\begin{enumerate}[(i)]
\item  The \emph{doubling} of $G$, denoted by $\hat{G}$, is a graph with the vertex set $V(\hat{G})=\{\hat{v}_1,\cdots,\hat{v}_n\}$, and a doubled edge set $E(\hat{G})=\{\hat{e}_{1,1},\hat{e}_{1,2},\hat{e}_{2,1},\hat{e}_{2,2},\cdots,\hat{e}_{m,1}, e_{m,2}\}$ where $\hat{e}_{k,1}$ and $\hat{e}_{k,2}$ are two duplicated edges of $e_k$, i.e., 
if the adjacent vertices of $e_k$ are $v_i$ and $v_j$, then the  the adjacent vertices of both $\hat{e}_{k,1}$ and $\hat{e}_{k,2}$ are $\hat{v}_i$ and $\hat{v}_j$. Any metric $\mu$ on $G$ induces naturally a metric $\hat{\mu}$ on $\hat{G}$ as $\hat{\mu}(\hat{e}_{k,1})=\hat{\mu}(\hat{e}_{k,2})=\mu(e_k)$. 
\item There is a natural \emph{projection map} $\pi$ from $\hat{G}$ to $G$, sending $\hat{v}_i$ to $v_i$ for all $v_i\in V(G)$, and sending $\hat{e}_{i,1}$ and $\hat{e}_{i,2}$ to $e_i$ for all $e_i\in E(G)$. Such a covering map induces maps, also denoted by $\pi$, from subgraphs of $\hat{G}$ to subgraphs of $G$, and from walks on $\hat{G}$ to walks on $G$. (Note that if $\hat{C}$ is a circuit on $\hat{G}$, then $\pi(\hat{C})$ is a closed walk but not necessarily a circuit on $G$.)
\item Let $\hat{\mfrako}$ be an orientation on $\hat{G}$ such that the edge pairs $\hat{e}_{k,1}$ and $\hat{e}_{k,2}$ are always oppositely oriented, i.e., if $\hat{\ve}_{\hat{\mfrako}, k,1}$ orients $v_i$ to $v_j$, then $\hat{\ve}_{\hat{\mfrako}, k,2}$ must orient from $v_j$ to $v_i$. A \emph{lift} of a task matrix $Q=(q_{ij})_{n\times n}$ of $G$ is an element $\alpha\in C_1(\hat{G},\mbbZ)$ such that $\alpha=\sum_{k=1}^m (c_{k,1}\cdot \hat{\ve}_{\hat{\mfrako}, k,1}+c_{k,2}\cdot\hat{\ve}_{\hat{\mfrako}, k,2})$ where $c_{k,1}=q_{ij}$ and  $c_{k,2}=q_{ji}$, knowing that $\hat{\ve}_{\hat{\mfrako}, k,1}$ orients from $v_i$ to $v_j$ and $\hat{\ve}_{\hat{\mfrako}, k,2}$ orients from $v_j$ to $v_i$. 
\end{enumerate}
\end{definition}

We now present our partial solution to 1c-TRP of type (1,1) as stated in  Algorithm~3 and Theorem~\ref{T:1c-TRP}. In particular, the result is built on the Shortest Vector Problem (SVP)~\cite{MG2002shortest, P2009public} on a subset of  the lattice $H_1(\hat{G},\mbbZ)$ using the $L^1$ norm $\Vert\cdot\Vert_\mu$. Note that the SVP in this setting can be solved with a ``fairly good basis'' of the lattice $H_1(\hat{G},\mbbZ)$: (1) Choose a spanning tree $\hat{T}$ of $\hat{G}$ and let $\{\hat{\ve}_1,\cdots,\hat{\ve}_{\hat{g}}\}=\vE_{\hat{\mfrako}}(\hat{G})\setminus E_{\hat{\mfrako}}(\hat{T})$ where $\hat{g}$ is the genus (or the first Betti number) of $\hat{G}$; (2) Let $\hat{C}_i$ be the unique circuit on $\hat{G}$ which is the concatenation of $\hat{\ve}_1$ and the unique path from $\hat{\ve}_1(1)$ to $\hat{\ve}_1(0)$ on $\hat{T}$ for all $i=1,\cdots,\hat{g}$; (3) $\{\hat{C}_1^{\ab},\cdots, \hat{C}_{\hat{g}}^{\ab}\}$ is a desired basis.

\begin{algorithm}
\label{A:1cTRP}
\caption{A partial solution to 1c-TRP of type (1,1) (Problem~\ref{Pr:1c-TRP-new})}
\begin{algorithmic}[1]
\State \textbf{Input:}  A simple graph $G$ with vertex set $V(G)=\{v_1,\cdots,v_n\}$ and edge set $E(G)=\{e_1,\cdots,e_m\}$, a metric $\mu$ of $G$,  a task matrix $Q=(q_{ij})_{n\times n}$ of $G$,  and $v\in V(G)$
\State \textbf{Output:} A closed walk on $G$ starting from vertex $v$
\State \textbf{Procedure:}
\State Obtain the doubling  graph $\hat{G}$ of $G$, the covering map $\pi$, a vertex $\hat{v}=\pi^{-1}(v)$, a metric $\hat{\mu}$ on $\hat{G}$, an orientation $\hat{\mfrako}$ of  $\hat{G}$, and a lift $\alpha\in C_1(\hat{G},\mbbZ)$ of $Q$ as in Definition~\ref{D:doubling}
\If{$\hat{v}\in \supp(\alpha)$}
\State Solve a Shortest Vector Problem: find a lattice point $\beta$ in $\Lambda_{\hat{G}}(\alpha)\bigcap H_1(\hat{G},\mbbZ)$ which minimizes $\Vert\beta\Vert_\mu$ 
\Else
\State Obtain $\Out_{\hat{\mfrako}}(\hat{v})=\{\hat{\ve}_1,\cdots, \hat{\ve}_s\}$
\For{$i=1,\cdots,s$}
\State $\alpha_i\gets \alpha+\hat{\ve}_i$
\State Solve a Shortest Vector Problem: find a lattice point $\beta_i$ in $\Lambda_{\hat{G}}(\alpha_i)\bigcap H_1(\hat{G},\mbbZ)$ which minimizes $\Vert\beta_i\Vert_\mu$
\EndFor
\State Choose $\beta\in\{\beta_1,\cdots,\beta_s\}$ such that $\Vert\beta\Vert_\mu = \min_{i=1}^s \Vert\beta_i\Vert_\mu$
\EndIf
\State Run Algorithm~2 for $\beta$ to derive a circuit $\hat{C}$ on $\hat{G}$
\State Translate $\hat{C}$ if necessary to make $\hat{C}$ start at $\hat{v}$ 
\State $C\gets \pi(\hat{C})$
\State \Return $C$ 
 \end{algorithmic}
 \end{algorithm}
 
 \begin{theorem} \label{T:1c-TRP}
If $\beta$ derived at Step~6 or Step~13  of Algorithm~3 is connected, then Algorithm~3 solves Problem~\ref{Pr:1c-TRP-new}: 
 \end{theorem}
 \begin{proof}
 First we note that each closed walk $C=\va_1\cdots\va_L$ on $G$ can be lifted to a circuit $\hat{C}=\hat{\va}_1\cdots \hat{\va}_L$ on $\hat{G}$ in the following way: if $\va_k$ orients from $v_i$ to $v_j$, then $\hat{\va}_k$ is the unique oriented edge from $\hat{v}_i$ to $\hat{v}_j$ with respect to $\hat{\mfrako}$. And in this sense, $\tau_\mu(C)=\tau_{\hat{\mu}}(\hat{C})$. Therefore, finding a closed walk $C$ on $G$ with initial vertex $v$ such that $\theta_{ij}(C)\geq q_{ij}$ for all adjacent $v_i$ and $v_j$ is equivalent to finding a shortest circuit $\hat{C}$ on $\hat{G}$ with initial vertex $\hat{v}$ such that $\hat{C}^{\ab}\geq \alpha$. Since $\hat{C}^{\ab}\in H_1(\hat{G},\mbbZ)$, this means $\hat{C}^{\ab}\in\Lambda_{\hat{G}}(\alpha)\bigcap H_1(\hat{G},\mbbZ)$. Note that $\Lambda_{\hat{G}}(\alpha)\bigcap H_1(\hat{G},\mbbZ)$ must be nonempty, since $\hat{G}$ is bridgeless. 
 
 Suppose $\hat{v}\in \supp(\alpha)$ at Step~5.  If $\beta$ derived at Step~6 is connected, then at Step~15, Algorithm~2 reduces to Algorithm~1, and the output $\hat{C}$ is a direction consistent circuit satisfying $\Vert\beta \Vert_{\hat{\mu}}=\tau_{\hat{\mu}}(\hat{C})$. Therefore, for all circuits $\hat{C'}$ such that $\hat{C'}^{\ab}\in\Lambda_{\hat{G}}(\alpha)\bigcap H_1(\hat{G},\mbbZ)$, by Lemma~\ref{L:ineq}, we know that $\hat{C'}$ goes through $\hat{v}$ and $\tau_{\hat{\mu}}(\hat{C'})\geq \Vert \hat{C'}^{\ab} \Vert_{\hat{\mu}}\geq \Vert\beta \Vert_{\hat{\mu}}=\tau_{\hat{\mu}}(\hat{C})$, solving Problem~\ref{Pr:1c-TRP-new}. 
 
 Now suppose $\hat{v}\notin \supp(\alpha)$. Then one can also observe that the desired shortest circuit $\hat{C}$ must satisfy  $\hat{C}^{\ab}\in\left(\bigcup_{i=1}^s\Lambda_{\hat{G}}(\alpha_i)\right)\bigcap H_1(\hat{G},\mbbZ)$. If $\beta$ derived at Step~6 is connected, then at Step~15, Algorithm~2 again reduces to Algorithm~1, and the output $\hat{C}$ is a direction consistent circuit satisfying $\Vert\beta \Vert_{\hat{\mu}}=\tau_{\hat{\mu}}(\hat{C})$. Now consider $\hat{C'}^{\ab}\in\Lambda_{\hat{G}}(\alpha_i)\bigcap H_1(\hat{G},\mbbZ)$ for some $1\leq i\leq s$, we know that $\hat{C'}$ goes through $\hat{v}$ and $\tau_{\hat{\mu}}(\hat{C'})\geq \Vert \hat{C'}^{\ab} \Vert_{\hat{\mu}}\geq \Vert\beta_i \Vert_{\hat{\mu}}\geq \Vert\beta \Vert_{\hat{\mu}} =\tau_{\hat{\mu}}(\hat{C})$, solving Problem~\ref{Pr:1c-TRP-new}. 
 
 \end{proof}
 
 \begin{remark}
 In general, if $\beta$ at Step~15 of Algorithm~3 is not connected, then the output of Algorithm~3 is not necessarily an optimal solution. The reason is that the extra paths in the Steiner tree derived when running Algorithm~2 introduces extra length to the output circuit. Actually for $\beta'\in \Lambda_{\hat{G}}(\alpha)\bigcap H_1(\hat{G},\mbbZ)$, we may derive a Steiner tree $T_{\beta'}$ specific to $\beta'$ ($T_{\beta'}$ is trivial if and only if $\beta'$ is connected). Running Algorithm~2 for $\beta'$, we get $\hat{C}_{\beta'}$ such that $\tau_\mu(\hat{C}_{\beta'})=2\mu(T_{\beta'})+\Vert\beta'\Vert_\mu$ by Theorem~\ref{T:HSCDP}. Because of  the extra length term $2\mu(T_{\beta'})$, optimal $\beta$ does not necessarily afford optimal $\hat{C}$ if $\beta$ is not connected. 
 \end{remark}

\bibliographystyle{alpha}
\bibliography{citation}
\end{document}